\documentclass{amsart}

\usepackage{amssymb}
\usepackage{comment}
\usepackage{soul}
\usepackage{mathrsfs}
\usepackage[stretch=10,shrink=10]{microtype}
\usepackage[showframe=false, margin = 1.5 in]{geometry}
\usepackage{mathtools}
\mathtoolsset{showonlyrefs} 
\usepackage[shortlabels]{enumitem}
\usepackage{ifthen}
\newcounter{dummy}
\usepackage{stmaryrd}
\usepackage{tikz,tikz-cd}
\usepackage{array}
\usepackage{color}
\usepackage{colortbl}
\usepackage{hyperref}
\hypersetup{
     colorlinks   = true,
     citecolor    = black
}
\usepackage{theoremref}
\usetikzlibrary{patterns}
\usetikzlibrary{patterns.meta}

\makeatletter
\def\thmref@flush{%
   \ifx\thmref@last\empty\else
      \ifthmref@comma, \thmref@finaltrue\fi \thmref@commatrue
      \thmref@last \ifx\thmref@stack\empty\else s\fi \thmref@num 0
      \let\do\thmref@one \thmref@stack
      \ifcase\thmref@num\or\space and\else\thmref@finaltrue, and\fi
      ~\ref{\thmref@head}\let\thmref@stack\empty\fi}
\def\thmref@one#1{\ifnum\thmref@num>0,\fi
   \space\ref{#1}\advance\thmref@num 1\relax}
\makeatother

\makeatletter
\newcommand\myitem[1][]{\item[#1]\refstepcounter{dummy}\def\@currentlabel{#1}}
\makeatother

\newcommand{\E}{\mathbf{E}}
\renewcommand{\P}{\mathbf{P}}
\newcommand{\1}{\mathbf{1}}
\newcommand{\f}{\frac}

\DeclareMathOperator{\Ber}{Bernoulli}
\DeclareMathOperator{\Geo}{Geo}

\DeclarePairedDelimiter\abs{\lvert}{\rvert}%
\DeclarePairedDelimiter\ii{\llbracket}{\rrbracket}%

\newcommand{\tikzdashedmid}{%
  \mathrel{\tikz[]
    \draw[dotted, line cap=round] (0,-0ex) -- (0,1.6ex);%
  }%
}

\newcommand{\erase}{\mathsf{GreedyPath}}
\newcommand{\Z}{\mathbb{Z}}
\newcommand{\N}{\mathbb{N}}
\newcommand{\instr}{\mathsf{instr}}
\newcommand{\eosos}[1]{\mathcal{O}_{#1}}
\newcommand{\s}{\mathfrak{s}}
\newcommand{\critical}{\rho}
\newcommand{\lpcritical}{\rho_*}

\newcommand{\rt}[2]{\mathcal{R}_{#2}(#1)}
\newcommand{\up}[2]{\mathcal{U}_{#2}(#1)}
\newcommand{\down}[2]{\mathcal{D}_{#2}(#1)}
\newcommand{\lt}[2]{\mathcal L_{#2}(#1)}
\newcommand{\emax}{e_{\max}}

\DeclareFontEncoding{LS1}{}{}
\DeclareFontSubstitution{LS1}{stix}{m}{n}
\DeclareSymbolFont{stixletters}{LS1}{stix}{m}{it}
\DeclareMathAccent{\cev}{\mathord}{stixletters}{"91}
\DeclareMathAccent{\vec}{\mathord}{stixletters}{"92}
\DeclareMathAccent{\vecev}{\mathord}{stixletters}{"95}
\newcommand{\greedy}[1][k]{\text{$k$-\erase}}
\makeatletter

\newcommand{\crist}[1][\@nil]{%
\def\tmp{#1}%
   \ifx\tmp\@nnil
       \lpcritical
    \else
       \lpcritical^{(#1)}
    \fi}

\newcommand{\Left}{\ensuremath{\mathtt{left}}}
\newcommand{\Right}{\ensuremath{\mathtt{right}}}
\newcommand{\Up}{\ensuremath{\mathtt{up}}}
\newcommand{\Down}{\ensuremath{\mathtt{down}}}
\newcommand{\Sleep}{\ensuremath{\mathtt{sleep}}}

\newtheorem{thm}{Theorem}[section]
\newtheorem{lemma}[thm]{Lemma}
\newtheorem{prop}[thm]{Proposition}

\newtheorem*{conjecture*}{Density Conjecture}

\theoremstyle{remark}

\theoremstyle{definition}
\newtheorem{define}[thm]{Definition}

\newcommand{\critcomb}[1]{\critical_{\mathtt{COMB}}(#1)}
\newcommand{\critComb}{\critical_{\mathtt{COMB}}}
\newcommand{\critinterval}{\critical_{\mathtt{INTERVAL}}}
\newcommand{\critTeeth}[1]{\critical_{\mathtt{TEETH}}\left(#1\right)}
\newcommand{\critSpine}[1]{\critical_{\mathtt{SPINE}}\left(#1\right)}
\newcommand{\critstarc}{\critical^*_{\mathtt{COMB}}}
\newcommand{\critstari}{\critical^*_{\mathtt{INTERVAL}}}
\newcommand{\comb}[1]{\mathfrak{C}_{#1}}
\newcommand{\combinterior}[1]{\comb{#1}^o}
\newcommand{\greedyinterval}[2]{\rho_{\mathtt{INTERVAL}}^{(#1)}\left(#2\right)}
\newcommand{\greedyInterval}[1]{\rho_{\mathtt{INTERVAL}}^{(#1)}}

\newcommand{\greedyComb}[1]{\rho_{\mathtt{COMB}}^{(#1)}}
\newcommand{\shape}[1]{\mathtt{shape}_{#1}}

\newcommand{\In}[2]{\mathtt{In}_{#1}(#2)}
\newcommand{\Out}[2]{\mathtt{Out}_{#1}(#2)}
\newcommand{\law}{\nu}
\newcommand{\lawI}{\nu_{\mathtt{I}}}
\newcommand{\lawC}{\nu_{\mathtt{C}}}
\newcommand{\lawCthreeD}{\nu_{\mathtt{Cd3}}}
\newcommand{\m}{\mathfrak{m}}

\newcommand{\filloffsetone}[2][]{
  \begin{scope}[shift={#2}]
    \fill[#1]   (0,0) rectangle +(1,1);
    \fill[#1](1,0) rectangle +(1,1);
    \fill[#1](2,0) rectangle +(1,1);
    \fill[#1] (1,1) rectangle +(1,1);
    \fill[#1](1,2) rectangle +(1,1);
    \fill[#1]  (0,2) rectangle +(1,1);
  \end{scope}
}

\newcommand{\fillcellwithoffset}[2]{
\begin{scope}[shift={#2}]
        \fill[#1] (0,0) rectangle ++(1,1);
    \end{scope}
}

\newcommand{\fillcellwithoffsetdiag}[2]{
    \begin{scope}[shift={#2}]
        \fill[pattern={Lines[angle=45, distance=4pt, line width=2pt]}, pattern color=#1] (0,0) rectangle ++(1,1);%
    \end{scope}
}

\newcommand{\fillcellwithoffsetnegdiag}[2]{
    \begin{scope}[shift={#2}]
        \fill[pattern={Lines[angle=135, distance=4pt, line width=2pt]}, pattern color=#1] (0,0) rectangle ++(1,1);%
    \end{scope}
}

\title{Activated random walk on the comb}
\author{Matthew Junge} \address{Matthew Junge, Department of Mathematics, Baruch College, City University of New York}
	\email{\texttt{matthew.junge@baruch.cuny.edu}}
\author{Josh Meisel}\address{Josh Meisel, Department of Mathematics, Graduate Center, City University of New York}
	\email{\texttt{jmeisel@gradcenter.cuny.edu}}
\author{Aldo Morelli}\address{Aldo Morelli, Department of Mathematics, University of New Mexico}
	\email{\texttt{amorelli@unm.edu}}

\begin{document}
\thanks{Junge was partially supported by NSF DMS Grant 2238272. Morelli was partially supported by NSF DMS Grants 2238272 and 2349366. Part of this research was conducted during the 2025 Baruch College Discrete Mathematics NSF Site REU}

\begin{abstract}
    The density conjecture for activated random walk on the interval was recently resolved using a new tool called layer percolation. As a step towards understanding how layer percolation extends to activated random walk on more complex graphs, we develop its analog for the comb graph and use it to prove bounds on the critical density. Additionally, we provide simulation evidence suggesting that the comb has different critical densities on its spine and teeth, both of which are smaller than the critical density for the interval. 
\end{abstract}

\maketitle

\section{Introduction}
Self-organized criticality (SOC) is an influential physics theory developed in the late 1980s to explain how a vast range of complex systems in nature maintain critical-like states \cite{BakTangWiesenfeld87}. Despite several candidate mathematical ``sandpile" models of SOC being proposed over the past four decades, rigorously establishing the predictions made by SOC in any of them has proven difficult \cite{watkins2016twentyfive, fey2010driving}. 
Activated random walk (ARW) is one such model that has begun to distinguish itself as particularly promising \cite{dickman2010activated, rolla2020activated, levine2023universality}. In 2024, Hoffman, Johnson, and Junge introduced a new tool for studying ARW in dimension one that they called {layer percolation}. It has since been used to establish several universality properties desirable in a model of SOC \cite{hoffman2024density, hoffman2024hockey, hoffman2025cutoff}. 

Understanding ARW, and by extension SOC, in more complicated geometries is an important line of inquiry. The main point of this article is to explore the nature of extending the layer percolation process introduced in \cite{hoffman2024density} beyond the interval.
We do so by generalizing it to include arbitrary shapes as its building blocks. This more general picture allows us to describe and study ARW on a comb graph, i.e.\ an interval with each of its vertices also connected to a distinct tooth vertex. Even this slight change introduces significant complexity. 

\subsection{Model definition}

ARW is a particle system on a graph consisting of \textit{active} and \textit{sleeping} particles. Active particles jump as rate-1 continuous-time random walks. Isolated active particles fall asleep at rate $\lambda>0$. A sleeping particle is dormant but becomes active the moment an active particle jumps onto its site. The graph may also contain special sites called \textit{sinks}, which kill particles that jump to them. 

The \textit{driven-dissipative} model of ARW consists of adding particles one at a time and then \emph{stabilizing} to a \emph{stable} configuration with only empty sites and isolated sleeping particles. Each new particle addition and subsequent stabilization is a time step of a Markov chain on the set of stable configurations. So long as $V$ is connected and contains a sink, the Markov chain has a unique stationary distribution that does not depend on the particular driving sequence \cite{levine2021exact}. In fact, the stationary distribution can be exactly sampled by placing an active particle at each site and then stabilizing. 

In \cite{hoffman2024density} it was proven that
\begin{align}
    \critinterval(\lambda) := \lim_{n\to\infty} \frac{\mathbf{E}[S'_n]}{n}, \label{eq:cI}
\end{align}
exists where $S'_n$ is the number of sleeping particles of a configuration sampled from the stationary distribution of driven-dissipative ARW on the interval $\ii{0,n+1} \subset \Z$, with the endpoints selected as sinks. This was done by developing a $(2+1)$-dimensional directed percolation model that they called {layer percolation}, which we will refer to as \textit{interval percolation}. %
Loosely speaking, this process tracks the number of instructions used in all balanced flows of particles when allowed to selectively wake sleeping particles without requiring a visit from an active particle. A superadditivity property inherent to interval percolation makes it easy to deduce that it has a critical value $\critstari$, which the main result of \cite{hoffman2024density} proved is equal to \eqref{eq:cI}.

We are interested in ARW on $\comb{n}$, the finite comb graph of length $n$. The subgraph $\ii{1, n}$ composes the \textit{spine}. There are sinks at $0$ and $n+1$ and one \textit{tooth} adjacent to each spine vertex (see Figure~\ref{fig:comb}). Let $S_n$ be the number of particles left sleeping on $\comb{n}$ when sampling from the stationary distribution of driven-dissipative ARW. Define the \textit{critical density for the comb} as
\begin{align*}
    \critcomb{\lambda} := \liminf_{n\to\infty} \frac{\mathbf{E}[S_n]}{n}.
\end{align*}
Note that $\critcomb{\lambda} \in [0,2]$ is the sum of the average densities on the spine and teeth. This is a natural viewpoint for one-dimensional graphs indexed by $\ii{1,n}$ and connects nicely with our methods.

\begin{figure} 
\centering
    \begin{tikzpicture}[scale = .75]
    \small 
	\draw (-3.9,0) -- (3.9,0);
	\foreach \x in { -3,...,3} {
            \draw (\x,0) -- (\x,1);
            \draw [fill=black] (\x,0) circle [radius=0.075];
            \draw [fill=black] (\x,1) circle [radius=0.075];
            \node at (\x, -0.5) {\the\numexpr\x+4\relax};
        }
        \node at (-4.1,0) {$\times$};
        \node at (-4.1,-.5) {$0$};
        \node at (4.1,0) {$\times$};
        \node at (4.1,-.5) {$8$};
    \end{tikzpicture}
    \vspace{-0.25cm}
    \caption{$\comb{7}$ with sinks denoted by $\times$. }
  \label{fig:comb}
\end{figure}
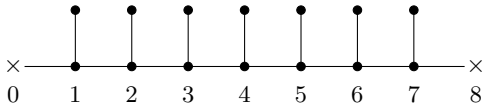

\subsection{Results}

 Our main contribution is developing an analog of interval percolation, which we call \textit{comb percolation}, that comes with a critical value $\critstarc$ related to $\critComb$. See Section~\ref{sec:mulp} for the definition.

\begin{thm}\thlabel{thm:cp}
It holds that $\critstarc(\lambda) \leq \critcomb{\lambda}$ for all $\lambda >0$.
\end{thm}
We believe that in fact $\critComb = \critstarc$ (see Section~\ref{sec:fw}). 
To prove \thref{thm:cp}, we start by generalizing interval percolation to a broader class of percolation processes parameterized by a probability distribution on sequences of shapes in $\mathbb N^d$. Interval and comb percolation can be viewed as special instances with distinct shape laws. This yields a new perspective on layer percolation by separating the ``building" from the ``bricks." 
Once comb percolation is constructed, we use ideas from \cite{hoffman2024density} and  \cite{hoffman2024hockey} to connect it to ARW and prove the lower bound in \thref{thm:cp}.

Our second result gives bounds on $\critComb$ in terms of $\critinterval$.
As explained at \eqref{eq:greed}, $\critstari$, and thus $\critinterval$, is the limit of the expected speed of greedy approximations $\critical^{(k)}_{\mathtt{INTERVAL}}$. Random scalings of these quantities are used in our lower bound. 
We write $\Geo(p)$ for the geometric distribution supported on the positive integers and $\Geo_0(p)$ to denote a geometric distribution supported on the nonnegative integers. 

\begin{thm}\thlabel{thm:bound}
    Given $\lambda>0$ and $K \sim \Geo(\lambda/(1+\lambda))$, it holds that
    \begin{align*}
        \frac{\lambda}{1+\lambda}\left(1+\mathbf{E}[K\greedyinterval{K}{3\lambda/2}]\right)\le  \critcomb{\lambda} \le 1+\critinterval(3\lambda/2).
    \end{align*}
\end{thm}

These bounds are interesting primarily because they connect $\critical_{\mathtt{COMB}}$ to $\critical_{\mathtt{INTERVAL}}$. The $3 \lambda /2$ comes from moves to the teeth creating more opportunities for particles to fall asleep on the spine (see \thref{lemma:IncreasedSleepRateCoupling}). Note that $\rho^{(k)}$ is easily seen to be strictly increasing in $\lambda$.
The lower bound is proven by constructing a renewal scheme in comb percolation that grows more slowly than $\critstarc$, and then invoking \thref{thm:cp}.  The upper bound is proven with a simple toppling scheme that uses  the well-known abelian property to realize a sample from ARW on the interval at stationarity after partially stabilizing the comb. Since it is known that $\critinterval(\lambda)<1$ for all $\lambda > 0$ \cite{HoffmanRicheyRolla20}, we deduce that $\critcomb{\lambda}<2$.

\subsection{Future work} \label{sec:fw}

The more general $\nu$-layer percolation process described in Section~\ref{sec:mulp} may have interesting intrinsic properties as well as potential applications to related processes such as the {stochastic sandpile model}. 
We also believe that much of the theory developed for the interval in \cite{hoffman2024density} as well as in Forien's work \cite{forien2025newproof} can be extended to ARW on the comb.
It appears possible, via an adaptation of the \textit{box lemma} from \cite{hoffman2024density} or the \textit{ejector seat} method from \cite{forien2025newproof}, to prove that $\limsup \E[S_n]/n \leq   \critstarc$ and so, by \thref{thm:cp}, the comb has a well-defined limiting critical density. We also believe that versions of the density, hockey-stick, and cutoff conjectures from \cite{hoffman2024density, hoffman2024hockey, hoffman2025cutoff} could be proven on the comb.

Let $T_n$ and $B_n$ be the number of sleeping particles on the teeth and spine, respectively, when sampling from the stationary distribution of the driven-dissipative ARW Markov chain on $\comb{n}$. Define $\critTeeth{\lambda}:= \lim_{n\to\infty} \mathbf{E}[T_n]/n$ and similarly $\critSpine{\lambda}:= \lim_{n\to\infty} \mathbf{E}[B_n]/n$. %
We conjecture that these limits exist and satisfy the relations 
\begin{align}
    \critTeeth{\lambda} < \critSpine{\lambda} < \critinterval(\lambda),  \label{eq:<}
\end{align}
supported by the simulation evidence in Figures~\ref{fig:sims} and ~\ref{fig:triplehockeystick}. 
This conjecture reveals a limitation of comb percolation, because, as far as we can tell, the methods from \cite{hoffman2024density} do not distinguish finely enough as to where particles in a stable configuration lie to tell us what happens on the teeth versus the spine. Our higher-dimensional approach in Section~\ref{sec:3d}, which does track this information, could provide insight on this matter.  

Another natural direction is to continue our effort to extend layer percolation to other geometries. One possible direction is to study one-dimensional graphs like the comb but with the teeth replaced by copies of some other graph $H$ attached to each spine vertex. For the comb, $H$ is a single vertex, and it appears that the important features of layer percolation are preserved. Exploring what happens with other $H$ could yield new insights into layer percolation and ARW. More complex one-dimensional graphs, such as those considered in \cite{ahlberg2015asymptotics} would be another interesting direction. Of course, we would also like to extend to higher dimensions. A natural next step would be the slab graph defined as the $k\times n$ subgraph of $\mathbb Z^2$ with $k$ fixed.

\begin{figure}

    \centering
    \includegraphics[scale = 0.6]{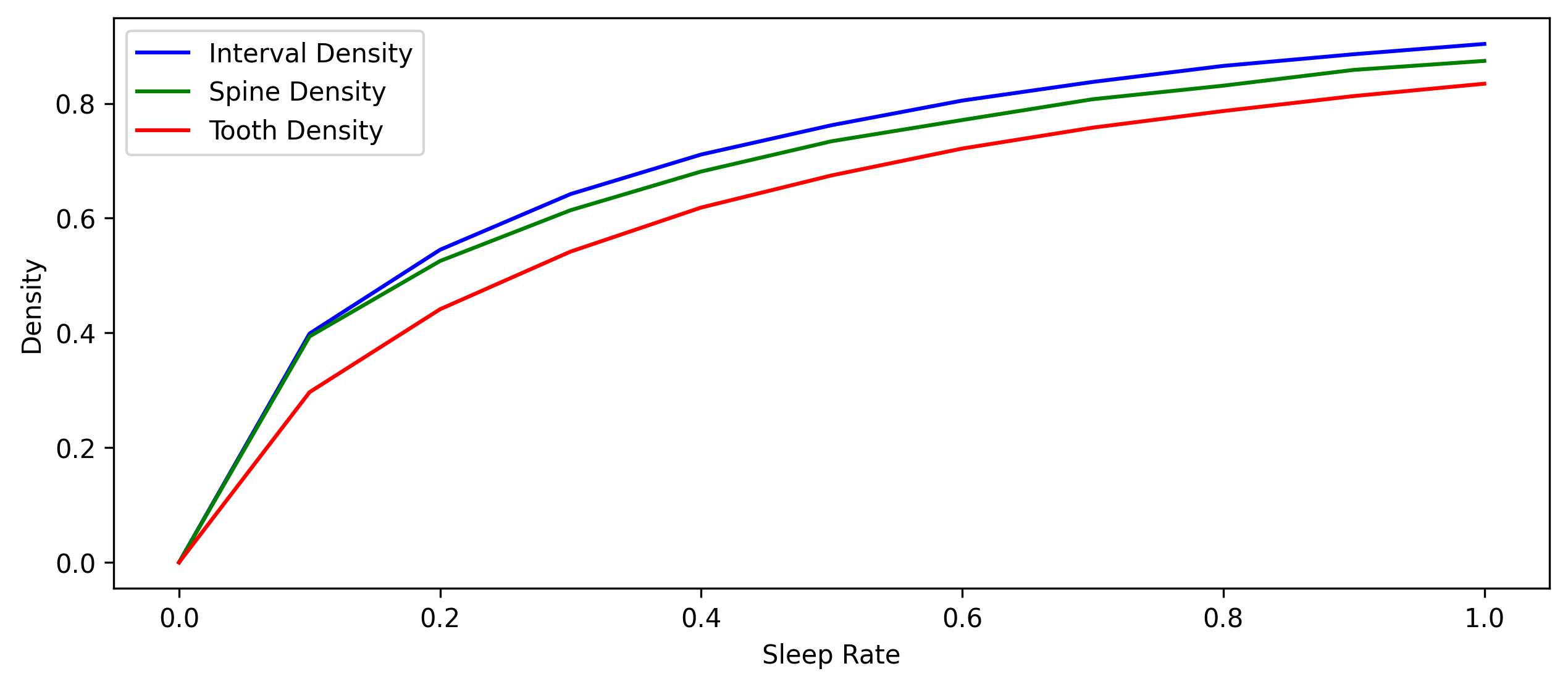}

  \caption{Plot of the empirical spine and teeth densities on $\comb{250}$ as well as the empirical density on the interval $\ii{1,250}$ for the respective driven-dissipative models at stationarity. The code used to run simulations and generate all plots in this paper may be found in this Github repository: \href{https://github.com/aldomorelli/Comb-ARW-Simulation-and-General-Layer-Percolation}{Comb ARW and Layer Percolation Simulation}.
  } 
  \label{fig:sims}
\end{figure}
\begin{figure}

    \centering
    \includegraphics[scale = 0.6]{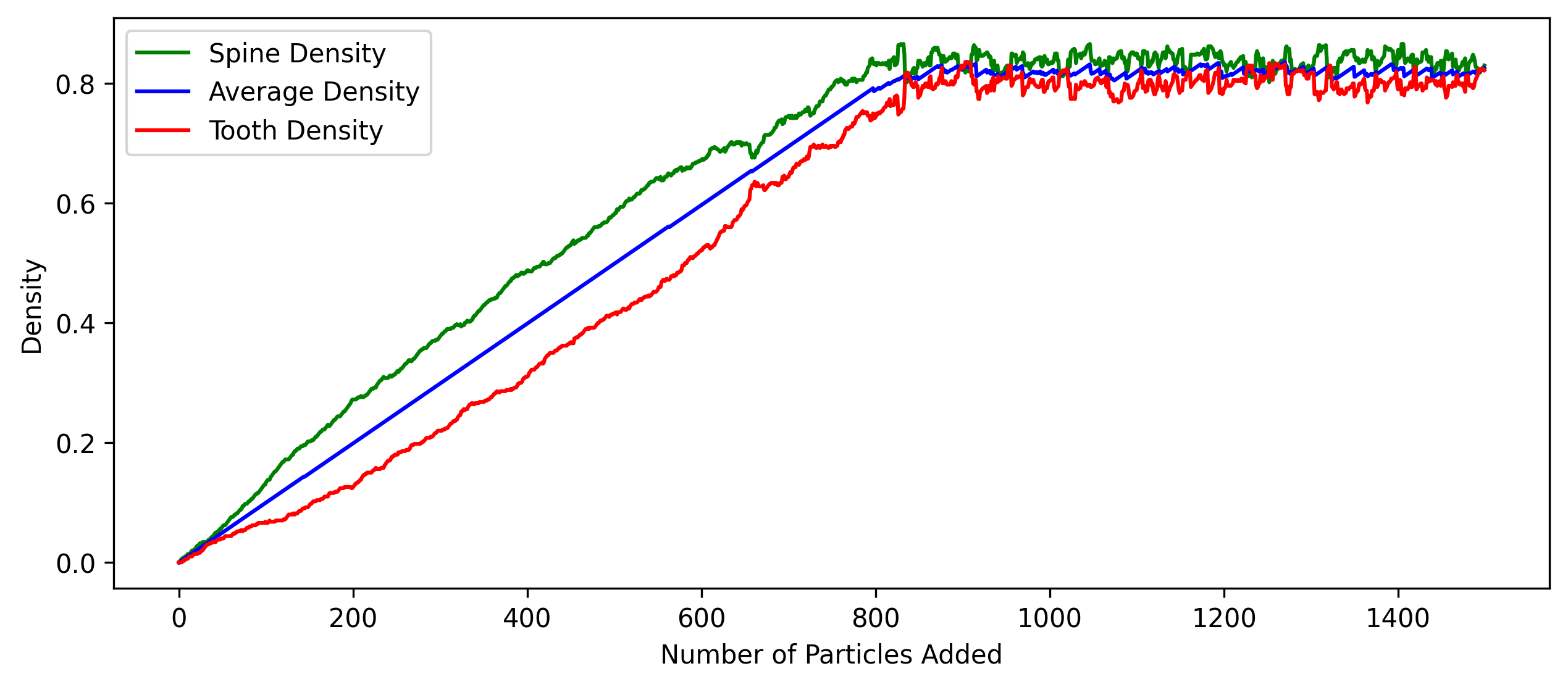}

  \caption{Plot of the spine and tooth densities over time in a sample path of driven-dissipative ARW on $\comb{500}$ with uniform driving and sleep rate $\lambda=0.8$, along with total average density on the comb. That the density graph converges to a piecewise linear-then-constant function for ARW on $\mathbb Z^d$ is known as the hockey-stick conjecture \cite{levine2023universality, hoffman2024hockey}. It is interesting to note that although the average of the tooth and spine densities appears to satisfy the hockey-stick conjecture, the individual densities do not. The fluctuations also appear larger for the separate densities but cancel out after averaging.}
  \label{fig:triplehockeystick}
\end{figure}

\subsection{Organization} 
In Section~\ref{sec:mulp} we define $\nu$-layer percolation then give a few examples which include interval and comb percolation. Section~\ref{sec:od} describes some basics about the representation of ARW that uses instructions at the sites and odometers that count the number of instructions executed. Section~\ref{sec:mulpo} connects Sections~\ref{sec:mulp} and \ref{sec:od} by describing how to translate between infection paths and odometers.  Section~\ref{sec:cp} proves \thref{thm:cp} and Section~\ref{sec:bound} proves \thref{thm:bound}. The appendix contains the statements of results we use from \cite{hoffman2024density} along with explanations as to why they continue to hold for comb percolation.

\section{\texorpdfstring{$\nu$}{nu}-layer percolation} \label{sec:mulp}

Non-empty, finite subsets of $\mathbb{N}^d$ are called \emph{$d$-dimensional shapes}. 
Given some probability measure $\law$ on sequences of $d$-dimensional shapes, we define \textit{$\nu$-layer percolation} to be a $(d+1)$-dimensional directed percolation process on $\mathbb N^d \times \mathbb N$. Elements of $\N^d \times \{k\}$ are called \textit{cells} at \textit{step $k$}, and following \cite{hoffman2024density} can be denoted $(r, s_1, s_2, \ldots, s_{d-1})_k$ to mean $(r, s_1, s_2, \ldots, s_{d-1}, k)$. Each cell at step $k$ infects a finite set of cells at step $k+1$ in the following manner. 

Let $(\shape{0}^k, \shape{1}^k, \ldots)$ be sampled according to $\law$. A cell $\mathbf{c} = (r, s_1, s_2, \ldots, s_{d-1})_k$ is said to be in \emph{diagonal} $j=r + s_1 + \ldots + s_{d-1}$, and its infection set will be a translation of $\shape{j}^k \times \{k+1\}$. For each $i \ge 0$, denote the maximal value of the $r$-coordinate of $\shape{i}^k$ by $r_i^k$. Then, letting $\ell^k_j=r^k_0+\dots+r^k_{j-1}$, the set of cells infected by $\mathbf c$ is given by 
$$
\left\{\left(\ell^k_j+r', s_1 + s'_1, \ldots, s_{d-1} + s'_{d-1}\right)_{k+1} \text{ such that } \left(r', s'_1, \ldots, s'_{d-1}\right) \in \shape{j}^k\right\}. 
$$
An \emph{infection path} starting from cell $\mathbf c_j$ is a sequence of cells $(\mathbf c_k)_{j\leq k \leq l}$ for some $j \le l \le \infty$ such that cell $\mathbf c_k$ infects cell $\mathbf c_{k+1}$ for all $j\leq k < l$. The \emph{infection set at step $k$} is the set of cells at step $k$ that belong to at least one infection path starting from $\mathbf 0 := (0,\hdots, 0)_0$. Let $\mathcal I$ denote the set of all infection paths starting from $\mathbf 0$. 

We define the \textit{height of $\nu$-layer percolation at step $k$} as 
\begin{align*}
    X_k := \max\{s_1+\dots+s_{d-1}\colon  (r,s_1, \hdots, s_{d-1})_k \text{ is infected by } \mathbf 0\}.
\end{align*}
A \textit{$k$-greedy path} is an infection path from $\mathbf 0$ to step $k$ that attains height $X_k$. As these paths are not necessarily unique, we will refer to \emph{the} $k$-greedy path as a uniformly sampled $k$-greedy path.
The \textit{expected speed of the $k$-greedy path} and the \emph{critical speed} are respectively defined as:
\begin{align}
    \rho^{(k)}:=\frac{1}{k} \mathbf{E}[X_k] \label{eq:greed}
\qquad \text{ and } \qquad 
    \rho^* := \lim_{k\to\infty} \rho^{(k)}. %
\end{align}
The limit $\rho^*$ exists since the $X_k$ are easily seen to be superadditive by concatenating greedy paths. 
We conclude this section with a few illustrative examples. The first two are simple starting examples to give the idea of layer percolation; they have no connection with the rest of the paper. The final two are interval and comb percolation. 

\subsection{1-percolation} Let $d=2$ and $\law= \law_1$ deterministically place probability one on sequences of a $3\times 3$ pixelated `1' shape whose bottom corner is (0,0).  For $\nu_1$-layer percolation, we have $X_k = 2k$ almost surely so that $\rho^*= 2$.  The sets of all cells infected by $\mathbf 0$ in the first three steps of $\law_1$-layer percolation are given in Figure~\ref{fig:onepercolation}.

\begin{figure}
  \centering
\begin{tikzpicture}[scale=.3]
\begin{scope}[shift={(0,0)}]
  \begin{scope}[shift={(0,0)}]
    \fill (0,0) rectangle +(1,1);
    \draw[thick,gray] (0,0) grid[step=1] (9,5);
  \end{scope}

\end{scope}
\begin{scope}[shift={(12,0)}]
  \fill[red] (0,0) rectangle +(1,1);
  \fill[orange] (1,0) rectangle +(1,1);
  \fill[yellow] (2,0) rectangle +(1,1);
  \fill[yellow] (1,1) rectangle +(1,1);
  \fill[yellow] (0,2) rectangle +(1,1);
  \fill[green] (1,2) rectangle +(1,1);
  \draw[thick,gray] (0,0) grid[step=1] (9,5);

\end{scope}
\begin{scope}[shift={(24,0)}]

  \filloffsetone[red]{(0,0)}
  \fillcellwithoffsetdiag{orange}{(2,0)}
  \fillcellwithoffset{orange}{(3,0)}
  \fillcellwithoffset{orange}{(3,1)}
  \fillcellwithoffset{orange}{(3,2)}
  \fillcellwithoffset{orange}{(2,2)}
  \fillcellwithoffset{orange}{(4,0)}
  \fillcellwithoffsetdiag{yellow}{(4,0)}
  \fillcellwithoffset{yellow}{(5,0)}
  \fillcellwithoffset{yellow}{(5,1)}
  \fillcellwithoffset{yellow}{(5,2)}
  \fillcellwithoffset{yellow}{(4,2)}
  \fillcellwithoffset{yellow}{(6,0)}
  \fillcellwithoffset{yellow}{(6,1)}
  \fillcellwithoffset{yellow}{(4,1)}
  \fillcellwithoffset{yellow}{(5,3)}
  \fillcellwithoffset{yellow}{(4,3)}
  \fillcellwithoffset{yellow}{(6,2)}
  \fillcellwithoffset{yellow}{(5,4)}
  \fillcellwithoffset{yellow}{(4,4)}
  \fillcellwithoffset{green}{(7,2)}
  \fillcellwithoffset{green}{(8,2)}
  \fillcellwithoffset{green}{(7,3)}
  \fillcellwithoffset{green}{(7,4)}
  \fillcellwithoffset{green}{(6,4)}
  \fillcellwithoffsetdiag{green}{(6,2)}
  
  \draw[thick,gray] (0,0) grid[step=1] (9,5);
\end{scope}

\draw[->, thick] (9.75,2.5) -- (11.25,2.5);
\draw[->, thick] (21.75,2.5) -- (23.25,2.5);
\end{tikzpicture}

  \caption{Infection sets at steps 0, 1, and 2 of $\nu_1$-layer percolation. The latter two steps are color-coded to illustrate infection across the steps. 
  }
  \label{fig:onepercolation}
\end{figure}
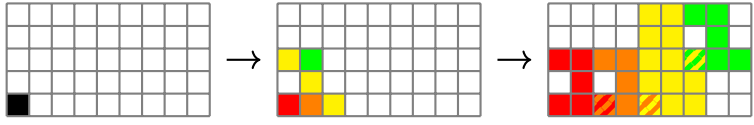

\subsection{Domino percolation} 
Let $d=2$ and $\nu= \nu_{\mathtt{D}}$, the law which each shape is sampled independently and uniformly from $\{(0,0), (0,1)\}$ and $\{(0,0),(1,0)\}$. A sample of the infection set of $\mathbf{0}$ in the first four steps along with step $250$ is given in Figure \ref{fig:dominopercolation}. We are not sure what the critical speed is.

\begin{figure}
  \centering
\begin{tikzpicture}[scale=.3]
\begin{scope}[shift={(0,0)}]
  \begin{scope}[shift={(0,0)}]
    \fill (0,0) rectangle +(1,1);
    \draw[thick,gray] (0,0) grid[step=1] (5,4);
  \end{scope}

\end{scope}
\begin{scope}[shift={(8,0)}]
    \fill (0,0) rectangle +(1,1);
    \fill (1,0) rectangle +(1,1);
  
  \draw[thick,gray] (0,0) grid[step=1] (5,4);

\end{scope}
\begin{scope}[shift={(16,0)}]
    \fillcellwithoffset{red}{(0,0)}
    \fillcellwithoffset{orange}{(1,0)}
    \fillcellwithoffset{orange}{(0,1)}
  
  \draw[thick,gray] (0,0) grid[step=1] (5,4);
\end{scope}
\begin{scope}[shift={(24,0)}]
    \fillcellwithoffset{red}{(0,0)}
    \fillcellwithoffset{red}{(1,0)}
    \fillcellwithoffsetdiag{orange}{(1,0)}
    \fillcellwithoffset{orange}{(1,1)}
    \fillcellwithoffset{orange}{(2,0)}
    \fillcellwithoffset{orange}{(2,1)}
  
  \draw[thick,gray] (0,0) grid[step=1] (5,4);
\end{scope}

\draw[->, thick] (5.75,2) -- (7.25,2);
\draw[->, thick] (13.75,2) -- (15.25,2);
\draw[->, thick] (21.75,2) -- (23.25,2);
\node at (31.25,2) {$\cdots$};
\end{tikzpicture}
\includegraphics[scale = 0.15]{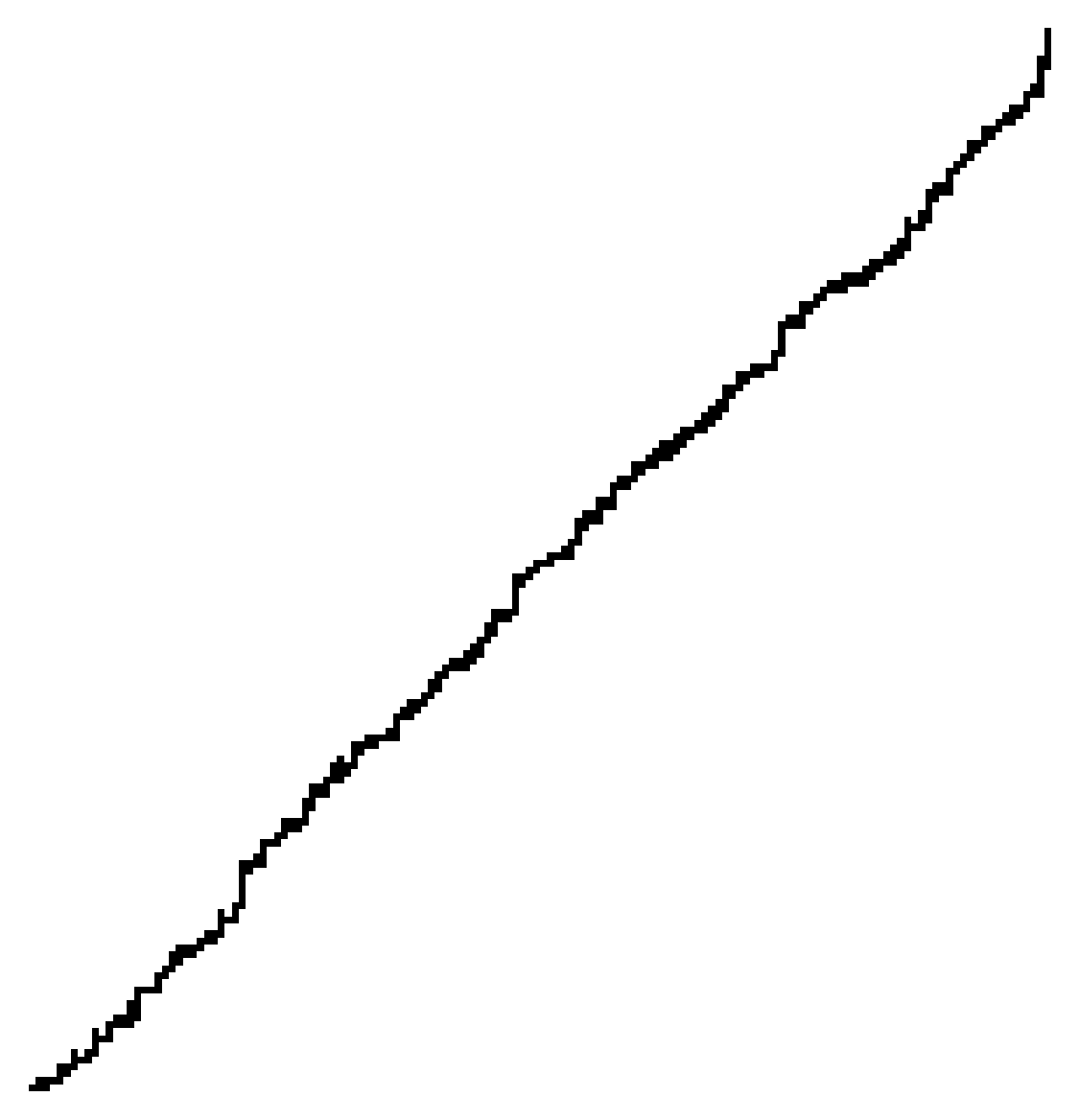}

  \caption{Infection sets at steps 0, 1, 2, 3 and 250 of a sample $\nu_{\mathtt{D}}$-layer percolation. In step $2$ the shape associated to both diagonals is $\{(0,0),(1,0)\}$. At step 250 we have $X_{250}= 140$. 
  }
  \label{fig:dominopercolation}
\end{figure}
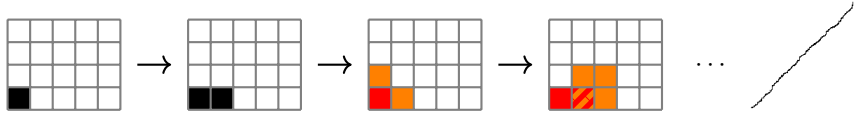

\subsection{Interval percolation} \label{sec:ip}
In the case of \textit{interval percolation} we have $d=2$, and under the law $\nu = \lawI(\lambda)$ the shapes are i.i.d. Each $\shape{j}$ has as its base the horizontal strip of cells starting at $(0,0)$ and with width $R_j \sim \Geo(1/2)$, that is $\{(i, 0) : 0 \le i \le R_j - 1\}$. For each base cell, the vertically adjacent cell is included independently with probability $\lambda/(1+\lambda)$ whose inclusion we track with independent random variables $B^j_i \sim \Ber(\lambda /(1+\lambda))$ for $i=1, \hdots, R_j$. The main theorem of \cite{hoffman2024density} proved that the critical speed for interval percolation is equal to $\critinterval$, and also to other natural formulations for a critical value of ARW. See Figure~\ref{fig:intervalperc} for an example of a few steps.

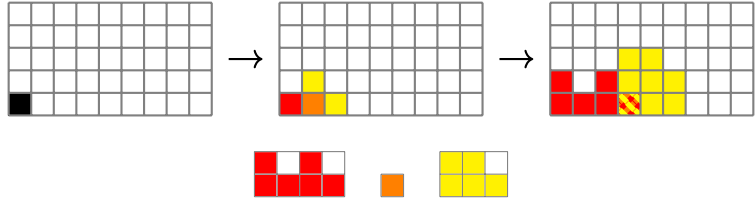
\begin{figure}
  \centering
\begin{tikzpicture}[scale=.3]
\begin{scope}[shift={(0,0)}]
  \begin{scope}[shift={(0,0)}]
    \fill (0,0) rectangle +(1,1);
    \draw[thick,gray] (0,0) grid[step=1] (9,5);
  \end{scope}

\end{scope}
\begin{scope}[shift={(12,0)}]
  \fill[red] (0,0) rectangle +(1,1);
  \fill[orange] (1,0) rectangle +(1,1);
  \fill[yellow] (1,1) rectangle +(1,1);
  \fill[yellow] (2,0) rectangle +(1,1);
  \draw[thick,gray] (0,0) grid[step=1] (9,5);

\end{scope}
\begin{scope}[shift={(24,0)}]
    \fillcellwithoffset{red}{(0,0)}  
    \fillcellwithoffset{red}{(0,1)}  
    \fillcellwithoffset{red}{(1,0)}  
    \fillcellwithoffset{red}{(2,0)}  
    \fillcellwithoffset{red}{(2,1)}  
    \fillcellwithoffset{red}{(3,0)}  
    \fillcellwithoffsetdiag{orange}{(3,0)}  
    \begin{scope}[shift={(3,0)}]
        \fill[pattern={Lines[angle=135, distance=4pt, line width=2pt]}, pattern color=yellow] (0,0) rectangle ++(1,1);
    \end{scope}
    \fillcellwithoffset{yellow}{(3,1)}  
    \fillcellwithoffset{yellow}{(4,0)}  
    \fillcellwithoffset{yellow}{(5,0)}  
    \fillcellwithoffset{yellow}{(4,1)} 
    \fillcellwithoffset{yellow}{(5,1)}  
    \fillcellwithoffset{yellow}{(3,2)}  
    \fillcellwithoffset{yellow}{(4,2)}

  \draw[thick,gray] (0,0) grid[step=1] (9,5);
\end{scope}

\draw[->, thick] (9.75,2.5) -- (11.25,2.5);
\draw[->, thick] (21.75,2.5) -- (23.25,2.5);
\end{tikzpicture}

\vspace{10pt}

\begin{center}
    \begin{tikzpicture}[scale=.3]

    \filldraw[fill=red, draw=white] (0,0) rectangle +(1,1);
    \filldraw[fill=red, draw=white] (1,0) rectangle +(1,1);
    \filldraw[fill=red, draw=white] (2,0) rectangle +(1,1);
    \filldraw[fill=red, draw=white] (0,1) rectangle +(1,1);
    \filldraw[fill=red, draw=white] (2,1) rectangle +(1,1);
    \filldraw[fill=red, draw=white] (3,0) rectangle +(1,1);

    \draw[gray] (0,0) grid[step=1] (4,2);
\end{tikzpicture}
\quad 
\begin{tikzpicture}[scale=.3]

    \filldraw[fill=orange, draw=white] (0,0) rectangle +(1,1);

    \draw[gray] (0,0) grid[step=1] (1,1);
\end{tikzpicture}
\quad 
\begin{tikzpicture}[scale=.3]

    \filldraw[fill=yellow, draw=white] (0,0) rectangle +(1,1);
    \filldraw[fill=yellow, draw=white] (1,0) rectangle +(1,1);
    \filldraw[fill=yellow, draw=white] (2,0) rectangle +(1,1);
    \filldraw[fill=yellow, draw=white] (0,1) rectangle +(1,1);
    \filldraw[fill=yellow, draw=white] (1,1) rectangle +(1,1);

    \draw[gray] (0,0) grid[step=1] (3,2);
\end{tikzpicture}
\end{center}

  \caption{Infection sets at steps 0, 1, and 2 in an example of interval percolation, along with the shapes associated to each diagonal in step 1.} 
  \label{fig:intervalperc}
\end{figure}

\subsection{Comb percolation} \label{sec:cpdef}

Again we take $d=2$. The law $\nu = \lawC(\lambda)$ of the shape sequences is more complicated than in interval percolation since there is dependence between adjacent shapes, and the shapes themselves are more involved. 

Each $\shape{j}$ in a sequence drawn from $\lawC(\lambda)$ consists of a random number $U_j\sim \Geo(1/2)$ of overlapping \textit{chunks}. Associated to each chunk $1 \le i \le U_j$ are the random variables $R_{j,i}\sim \Geo(2/3)$, $T_{j,i}\sim \Ber(\lambda/(1+\lambda))$ and $S_{j,i,\ell}\sim \Ber(\lambda/(1+\lambda))$ for $1\le \ell\le R_{j,i}$. The base of the $i^\text{th}$ chunk is a horizontal strip of cells of length $R_{j,i}$. Above the $\ell^\text{th}$ cell we will include a vertical strip of cells of height $T_{j,i}+S_{j,i,\ell}$. All $U_j, R_{j,i}, T_{j,i}, S_{j,i,\ell}$ are independent with one exception: for $j \neq 0$, $T_{j,1}$ is \textit{inherited} from the last chunk of the previous shape, that is $T_{j,1}$ is set to equal $T_{j-1,U_{j-1}}$. 
Once we have finished constructing our chunks, we concatenate them horizontally, beginning at $(0,0)$, and overlapping one column each time, giving our final shape. The critical speed $\critstarc$ for $\lawC$-percolation is the value referenced in \thref{thm:cp}.

Here is an example of how one might build $\shape{0}$. Suppose that $U_0 = 3$, so the shape is composed of three chunks. For the first chunk say we have $R_{0,1}=3, T_{0,1}=1,$ and $S_{0,1,1}=0$, $S_{0,1,2}=1$, $S_{0,1,3}=0$.  For the second chunk say we have $R_{0,2}=2$, $T_{0,2}=0$, and $S_{0,2,1}=1$, $S_{0,2,2}=0$. For the last chunk, suppose $R_{0,3}=2$, $T_{0,3}=1$, and $S_{0,3,1}=0$, $S_{0.3,2}=1$. Then the three chunks and resulting $\shape{0}$ are, respectively, 
\begin{center}
\begin{tikzpicture}[scale=.3]
\begin{scope}[shift={(0,0)}]
  \begin{scope}[shift={(0,0)}]
    \filldraw[fill=red, draw=white] (0,0) rectangle +(1,1);
    \filldraw[fill=red, draw=white] (1,0) rectangle +(1,1);
    \filldraw[fill=red, draw=white] (2,0) rectangle +(1,1);
    \filldraw[fill=red, draw=white] (0,1) rectangle +(1,1);
    \filldraw[fill=red, draw=white] (1,1) rectangle +(1,1);
    \filldraw[fill=red, draw=white] (2,1) rectangle +(1,1);
    \filldraw[fill=red, draw=white] (1,2) rectangle +(1,1);
    \draw[gray] (0,0) grid[step=1] (4,3);
  \end{scope}

  \begin{scope}[shift={(6,0)}]
    \filldraw[fill=orange, draw=white] (0,0) rectangle +(1,1);
    \filldraw[fill=orange, draw=white] (1,0) rectangle +(1,1);
    \filldraw[fill=orange, draw=white] (0,1) rectangle +(1,1);
    \draw[gray] (0,0) grid[step=1] (4,3);
  \end{scope}

  \begin{scope}[shift={(12,0)}]
    \filldraw[fill=yellow, draw=white] (0,0) rectangle +(1,1);
    \filldraw[fill=yellow, draw=white] (1,0) rectangle +(1,1);
    \filldraw[fill=yellow, draw=white] (0,1) rectangle +(1,1);
    \filldraw[fill=yellow, draw=white] (1,1) rectangle +(1,1);
    \filldraw[fill=yellow, draw=white] (1,2) rectangle +(1,1);
    \draw[gray] (0,0) grid[step=1] (4,3);
  \end{scope}

\end{scope}
\end{tikzpicture}
\qquad 
\begin{tikzpicture}[scale=.3]
\begin{scope}[shift={(0,0)}]
  \begin{scope}[shift={(0,0)}]
    \filldraw[fill=red, draw=white] (0,0) rectangle +(1,1);
    \filldraw[fill=red, draw=white] (1,0) rectangle +(1,1);
    \filldraw[fill=red, draw=white] (2,0) rectangle +(1,1);
    \filldraw[fill=red, draw=white] (0,1) rectangle +(1,1);
    \filldraw[fill=red, draw=white] (1,1) rectangle +(1,1);
    \filldraw[fill=red, draw=white] (2,1) rectangle +(1,1);
    \filldraw[fill=red, draw=white] (1,2) rectangle +(1,1);

    \filldraw[
    draw=white,
    pattern={Lines[angle=45, distance=4pt, line width=2pt]},
    pattern color=orange 
    ] (2,0) rectangle +(1,1);
    \filldraw[fill=orange, draw=white] (3,0) rectangle +(1,1);
    \filldraw[
    draw=white,
    pattern={Lines[angle=45, distance=4pt, line width=2pt]},
    pattern color=orange 
    ] (2,1) rectangle +(1,1);

    \filldraw[
    draw=white,
    pattern={Lines[angle=45, distance=4pt, line width=2pt]},
    pattern color=yellow 
    ] (3,0) rectangle +(1,1);
    \filldraw[fill=yellow, draw=white] (4,0) rectangle +(1,1);
    \filldraw[fill=yellow, draw=white] (3,1) rectangle +(1,1);
    \filldraw[fill=yellow, draw=white] (4,1) rectangle +(1,1);
    \filldraw[fill=yellow, draw=white] (4,2) rectangle +(1,1);
    \draw[gray] (0,0) grid[step=1] (6,4);
  \end{scope}

\end{scope}
\end{tikzpicture}
\end{center}
If we were to continue this example and now build $\shape{1}$, because $T_{0,3}=1$ we would have that the $T_{1,1}$ is also $1$, via inheritance. An example of the first few steps of $\lawC$-layer percolation is given in Figure \ref{fig:combperc}.
\begin{figure}
  \centering
\begin{tikzpicture}[scale=.3]
\begin{scope}[shift={(0,0)}]
  \begin{scope}[shift={(0,0)}]
    \fill (0,0) rectangle +(1,1);
    \draw[thick,gray] (0,0) grid[step=1] (9,5);
  \end{scope}

\end{scope}
\begin{scope}[shift={(12,0)}]
  \fill[red] (0,0) rectangle +(1,1);
  \fill[orange] (1,0) rectangle +(1,1);
  \fill[yellow] (1,1) rectangle +(1,1);
  \fill[yellow] (2,0) rectangle +(1,1);
  \fill[green] (2,1) rectangle +(1,1);
  \fill[green] (1,2) rectangle +(1,1);
  \draw[thick,gray] (0,0) grid[step=1] (9,5);

\end{scope}
\begin{scope}[shift={(24,0)}]
    \fillcellwithoffset{red}{(0,0)}  
    \fillcellwithoffset{red}{(0,1)}  
    \fillcellwithoffset{red}{(1,0)}  
    \fillcellwithoffset{red}{(1,1)}  
    \fillcellwithoffsetdiag{orange}{(1,0)}  
    \fillcellwithoffsetdiag{orange}{(1,1)}  
    \fillcellwithoffset{orange}{(1,2)}  
    \fillcellwithoffset{orange}{(2,0)}  
    \fillcellwithoffset{orange}{(3,0)}  
    \fillcellwithoffset{orange}{(3,1)}  
    \fillcellwithoffsetdiag{yellow}{(3,0)}  
    \fillcellwithoffsetdiag{yellow}{(3,1)}  
    \fillcellwithoffset{yellow}{(4,0)}  
    \fillcellwithoffset{yellow}{(4,1)}  
    \fillcellwithoffset{yellow}{(4,2)}   
    \fillcellwithoffset{yellow}{(3,2)}  
    \fillcellwithoffset{yellow}{(4,3)}  
    \fillcellwithoffsetdiag{green}{(4,1)}
    \fillcellwithoffsetdiag{green}{(4,2)}    
    \fillcellwithoffset{green}{(5,1)}  
    \fillcellwithoffset{green}{(6,1)}  
    \fillcellwithoffset{green}{(7,1)}  
    \fillcellwithoffset{green}{(7,2)}  
    \fillcellwithoffset{green}{(5,2)}  
    \fillcellwithoffset{green}{(6,2)}  
    \fillcellwithoffsetdiag{green}{(7,2)}  
    \fillcellwithoffset{green}{(7,3)}  
    \fillcellwithoffsetdiag{green}{(4,3)}  
    
  \draw[thick,gray] (0,0) grid[step=1] (9,5);
\end{scope}

\draw[->, thick] (9.75,2.5) -- (11.25,2.5);
\draw[->, thick] (21.75,2.5) -- (23.25,2.5);
\end{tikzpicture}

\vspace{10pt}

\begin{center}
    \begin{tikzpicture}[scale=.3]

    \filldraw[fill=red, draw=white] (0,0) rectangle +(1,1);
    \filldraw[fill=red, draw=white] (1,0) rectangle +(1,1);
    \filldraw[fill=red, draw=white] (1,1) rectangle +(1,1);
    \filldraw[fill=red, draw=white] (0,1) rectangle +(1,1);

    \draw[gray] (0,0) grid[step=1] (2,2);
\end{tikzpicture}
\quad 
\begin{tikzpicture}[scale=.3]

    \filldraw[fill=orange, draw=white] (0,0) rectangle +(1,1);
    \filldraw[fill=orange, draw=white] (1,0) rectangle +(1,1);
    \filldraw[fill=orange, draw=white] (0,1) rectangle +(1,1);
    \filldraw[fill=orange, draw=white] (0,2) rectangle +(1,1);
    \filldraw[fill=orange, draw=white] (2,0) rectangle +(1,1);
    \filldraw[fill=orange, draw=white] (2,1) rectangle +(1,1);

    \draw[gray] (0,0) grid[step=1] (3,3);
\end{tikzpicture}
\quad 
\begin{tikzpicture}[scale=.3]

    \filldraw[fill=yellow, draw=white] (0,0) rectangle +(1,1);
    \filldraw[fill=yellow, draw=white] (1,0) rectangle +(1,1);
    \filldraw[fill=yellow, draw=white] (1,2) rectangle +(1,1);
    \filldraw[fill=yellow, draw=white] (0,1) rectangle +(1,1);
    \filldraw[fill=yellow, draw=white] (1,1) rectangle +(1,1);

    \draw[gray] (0,0) grid[step=1] (2,3);
\end{tikzpicture}
\quad 
\begin{tikzpicture}[scale=.3]

    \filldraw[fill=green, draw=white] (0,0) rectangle +(1,1);
    \filldraw[fill=green, draw=white] (0,1) rectangle +(1,1);
    \filldraw[fill=green, draw=white] (1,0) rectangle +(1,1);
    \filldraw[fill=green, draw=white] (2,0) rectangle +(1,1);
    \filldraw[fill=green, draw=white] (3,0) rectangle +(1,1);
    \filldraw[fill=green, draw=white] (3,1) rectangle +(1,1);

    \draw[gray] (0,0) grid[step=1] (4,2);
\end{tikzpicture}
\end{center}

  \caption{Infection sets at steps 0, 1, and 2 of an example $\lawC$-layer percolation along with the shapes associated to each diagonal in step 1.} 
  \label{fig:combperc}
\end{figure}

In concatenating the chunks together with overlap, the cumulative contribution to each column of the $S_{j,i,\ell}$ indicators has distribution $\Ber\left(\tfrac{3\lambda/2}{1+3\lambda/2}\right)$. This leads to a coupling between the comb percolation and interval percolation processes which is formalized in the following lemma. 

\begin{lemma}\thlabel{lemma:IncreasedSleepRateCoupling} 
    There exists a coupling between $(\shape{j})_{j \ge 0} \sim \lawC(\lambda)$ and $(\shape{j}')_{j \ge 0} \sim \lawI(3\lambda/2)$ such that $\shape{j}'\subseteq \shape{j}$ for all $j\ge 0$.   
\end{lemma}

\begin{proof}
    Take a sequence $(\shape{j})_{j \ge 0} \sim \lawC(\lambda)$, and denote the width of $\shape{j}$ by $W_j := 1 + \sum_{i=1}^{U_j} R_{j,i}$, meaning $W_j - 1$ is its rightmost column. Let $\shape{j}'$ consist of its \emph{base} --- the horizontal strip of cells from $(0,0)$ to $(W_j-1, 0)$ --- along with the cells $(i, 1)$ for each $1\le i\le W_j-1$ such that at least one of the $S_{j,i,\ell}$ indicators associated to that column $i$ equals $1$. Clearly $\shape{j}'\subseteq \shape{j}$, and so we must show that $(\shape{j}')_{j \ge 0}$ is distributed as $\lawI(3\lambda/2)$.   

    First, $W_j$ indeed has distribution $\Geo(1/2)$, being one more than the sum of $\Geo(1/2)$ different independent $\Geo_0(2/3)$ random variables. 
    Next, we may think of building $\shape{j}$ column-wise, where at each column we have a chance of starting a new chunk and staying on the current column or continuing our existing chunk into the next column. Because the width of each chunk has distribution $\Geo(2/3)$ and the number of chunks has distribution $\Geo(1/2)$, the probability that we start a new chunk at each iteration of this process is $1/3$. As such the number of chunks associated to each column is distributed as $\Geo(2/3)$, and for each of these chunks we have an associated $S_{j,i, \ell}$ indicator which is one with probability $\lambda/(1+\lambda)$. If we denote the number of chunks associated to this column by $C_N$, the probability that at least one of these indicators equals $1$ is 
    \begin{align*}
        \sum_{n=1}^\infty \mathbb{P}(C_N=n)\left(1-\frac{1}{(1+\lambda)^{n}}\right) &= \frac{2}{3}\sum_{n=1}^\infty \frac{1}{3^{n-1}}\left(1-\frac{1}{(1+\lambda)^{n}}\right)
        =\frac{3\lambda/2}{1+3\lambda/2},
    \end{align*}
    and so each cell vertically adjacent to the base of $\shape{j}'$ is independently included in $\shape{j}'$ with probability $\tfrac{3\lambda/2}{1+3\lambda/2}$, which is identical to the inclusion of such a cell in a shape in a sequence generated from $\lawI(3\lambda/2)$. 
\end{proof}

\subsection{Comb percolation with \texorpdfstring{$d=3$}{d=3}} \label{sec:3d}
The stable odometers on the comb may also be represented by a layer percolation process on $\mathbb{N}^3\times \mathbb{N}$ with a law $\lawCthreeD$. The new law is identical to that of $\lawC$ except for when sampling shapes, the contribution of the $T_i$ and $S_{i,\ell}$ indicators are separated into their own dimensions. More explicitly, chunks are built from a base strip of cells from $(0,0,0)$ to $(R_i, 0,0)$, and for each cell in this strip we also include the cells $(\ell, 0, T_i)$, $(\ell, S_{i,\ell}, 0)$, and $(\ell, S_{i,\ell}, T_i)$. Chunks are then concatenated with one cell of overlap in the $r$ coordinate to form shapes. See Figure~\ref{fig:3dcombplot} for a simulation of the infection set.

We believe that $\lawCthreeD$-layer percolation captures the stable odometers on the comb (Section \ref{sec:generalstableodometers}) in the same way as comb percolation, with the added benefit that the coordinates $s_1, s_2$ represent the number of particles left sleeping on the spine and teeth respectively, and hence information about these individual densities is retained. In the following arguments however we will consider comb percolation in order to better draw connections to interval percolation and to utilize the tools from \cite{hoffman2024density, hoffman2024hockey}. 

\begin{figure}

    \centering
    \includegraphics[scale = 0.25]{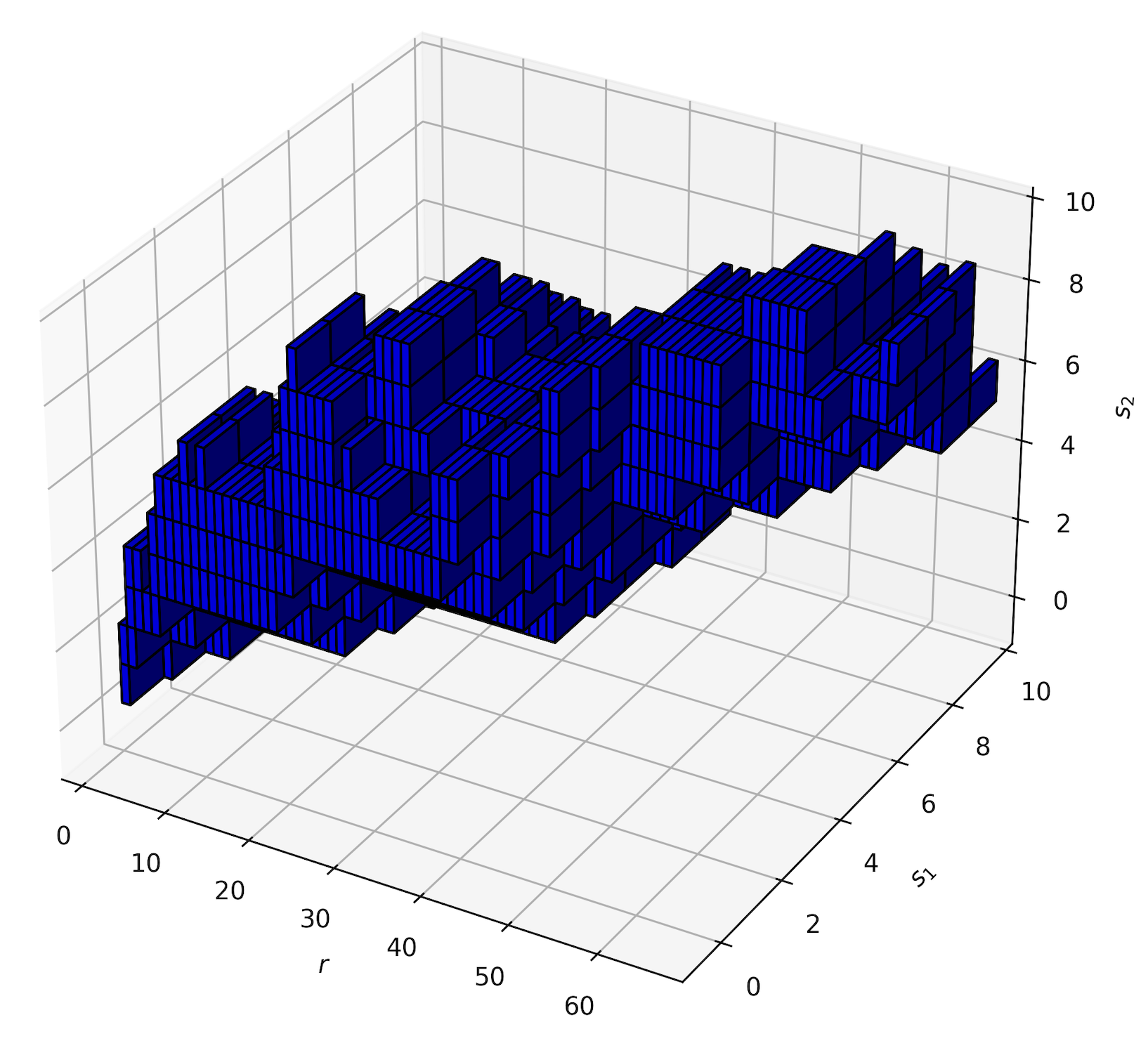}
  \caption{A simulated infection set at step 15 of $\lawCthreeD$-layer percolation with sleep rate $\lambda=0.5$.}
  \label{fig:3dcombplot}
\end{figure}

\section{Odometers}  \label{sec:od}
It is standard to describe ARW using the \textit{sitewise representation}, where rather than particles moving according to their own internal random walks, particles jump according to predetermined random instruction stacks associated to each site. Throughout the rest of this paper spine vertices will be labeled by their location in $\ii{0,n+1}$, and tooth vertices will be labeled as the prime of the spine vertex they attach to i.e., $1'$ is the tooth connected to $1$.

\subsection{Basics} With $\combinterior{n}$ referring to the subgraph of $\comb{n}$ consisting of the non-sink vertices, we associate to each $v \in \combinterior{n}$ a double-sided stack of instructions $\instr_v=\left(\instr_v\left(k\right)\right)_{k\in\mathbb{Z}}$. For vertices $v$ on the spine, $\instr_v(k)\in \{\Sleep, \Left, \Up, \Right\}$, where $\Sleep$ tells a particle to fall asleep and $\Left, \Up, \Right$ tell the particle to jump to $v-1, v', $ and $v+1$ respectively. For $v'$ on the teeth, $\instr_{v'}(k)\in \{\Sleep, \Down\}$ where $\Down$ tells a particle to jump to its neighbor $v$ on the spine. We require $\instr_v(0)=\Left$ for $v$ on the spine and $\instr_{v'}(0)=\Down$ for $v'$ on the teeth. All instructions $\instr_v(k)$ for $k \neq 0$ are independent. The probability of a $\Sleep$ instruction is $\lambda/(1+\lambda)$ at every vertex $v$, and otherwise a jump instruction is sampled uniformly from $\{\Left, \Up, \Right\}$ or $\{\Down\}$ depending on whether $v$ is on the spine or teeth.

An \textit{odometer} is any function $u\colon \comb{n} \to \mathbb{Z}$ that is zero at the sinks. We think of $u$ as counting the number of instructions used at each site.
Negative odometer values, a somewhat confusing notion that is needed to regularize comb percolation, can be thought of as representing the execution of instructions on the negative-index portion of the instruction list. Such instructions have the reverse of their normal effect; for example, executing $\instr_v(k) = \Left$ for some $k \le 0$ moves a particle from $v-1$ to $v$. Denote the number of $\Left$ instructions used at $v \in \combinterior{n}$ by
\[
\lt{u}{v} := 
\begin{cases}
\phantom{-}\displaystyle\sum\nolimits_{1 \le i \le u(v)} \mathbf{1}\{\instr_v(i) = \Left\}, & \text{if } u(v) \ge 0, \\[10pt]
-\displaystyle\sum\nolimits_{u(v) \le i \le 0} \mathbf{1}\{\instr_v(i) = \Left\}, & \text{if } u(v) < 0,
\end{cases}
\]
and for convenience let $\lt{u}{v} := 0$ at the sinks. Similarly define $\rt{u}{v}$, $\up{u}{v}$, and $\down{u}{v}$ to denote the respective number of $\Right$, $\Up$, and $\Down$ instructions used under the odometer at $v$.

A \textit{configuration} is a function $\sigma\colon \comb{n}\to \mathbb{N}\cup \{\s\}$ for which $\sigma(v)$ represents the initial number and type of particles at site $v$. If $\sigma(v)=\s$ we think of the site as containing a single sleeping particle, and define $|\s|=1$.
\textit{Stable} configurations satisfy $\sigma(v)\in \{0, \s\}$ for all $v \in \combinterior{n}$. 
A stabilizing procedure consists of sequentially selecting an unstable site $v$, i.e.\ a site $v \in \combinterior{n}$ with $\sigma(v)\not\in \{0, \s\}$, and executing the next (positively-indexed) instruction there until a stable configuration is achieved. The oft-used \textit{abelian property} of ARW implies that every stabilizing procedure results in the same odometer and stable configuration \cite{rolla2020activated}. Call this unique odometer the \textit{true odometer}. 

\subsection{Stable odometers} \label{sec:generalstableodometers}
\textit{Stable odometers} are a subclass of odometers that satisfy certain \emph{mass balance equations} for particle flow into and out of sites.
Given any odometer $u$, initial configuration $\sigma$, and instructions $\instr = (\instr_v)_{v \in \comb{n}}$, we define the net inflow and outflow of particles at each non-sink site under $u$. For $v'$ on the teeth, let 
\begin{align*}
    \In{v'}{u} := \up{u}{v} \qquad \text{ and } \qquad  \Out{v'}{u}:= \down{u}{v'}
\end{align*}
and similarly for $v$ on the spine, let
\begin{align*}
    \In{v}{u}:=  \rt{u}{v-1} + \lt{u}{v+1} + \down{u}{v'}\qquad \text{ and } \qquad  \Out{v}{u}:=\rt{u}{v}+\lt{u}{v}+\up{u}{v}.
\end{align*}
We call an odometer $u$ \textit{stable} under $\sigma,\instr$ if for all $v\in \combinterior{n}$,
\begin{enumerate}[(a)]
    \item $h(v):= |\sigma(v)|+\In{v}{u}-\Out{v}{u}\in \{0,1\}$
    \item $h(v)=1$ if and only if $\instr_v(u(v))=\Sleep$.
\end{enumerate}
These conditions require that (a) the flow into and out of $v$ be balanced, and (b) when a particle is left behind, the last instruction used is $\Sleep$. The true odometer is an example of a stable odometer, but there are many others. Note that these balance properties are still satisfied by stable odometers that take negative values, keeping in mind that negative-index instructions have the reverse effect of positive-index instructions. We will consider certain subclasses of stable odometers satisfying a flow condition to the left sink. 

\begin{define} For $f_0 \in \Z$, define $\eosos{n}(\instr, \sigma,f_0)$ to be the set of all {odometers} $u$ on $\comb{n}$ that are stable under $\instr,\sigma$ and have a net flow of $f_0$ from site~$0$ to site~$1$ i.e., $-\lt{u}{1}=f_0$.
\end{define}

A crucial bridge between infection paths in comb percolation and stable odometers is the \emph{minimal odometer} $\m$. It is the smallest odometer in $\eosos{n}(\instr,\sigma,f_0)$ in the sense that $\m(v) \leq u(v)$ for all $u \in \eosos{n}(\instr,\sigma,f_0)$ and $v \in \comb{n}$.  Note that the minimal odometer typically takes some negative values, so is almost never the true odometer, which must be nonnegative. Following the ideas of \cite{hoffman2024density}, we construct it inductively as follows.

\begin{define}[Minimal odometer $\m \in \eosos{n}(\instr,\sigma,f_0)$] \thlabel{def:m}
Set $\m(0)=0=\m(n+1)$. Now suppose that $\m(v-1)$ has already been defined for $v$ on the spine. We then define
$\m(v)$ to be the minimum integer such that $\lt{\m}{v} = \rt{\m}{v-1} - f_0 - \sum_{i=1}^{v-1}(\abs{\sigma(i)} + \abs{\sigma{(i')}}).$ For $v'$, we define $\m(v')$ to be the minimum integer such that $\down{\m}{v'}=\up{\m}{v}+|\sigma(v')|$. 
\end{define}

The odometer $\m$ as defined above can be seen to be minimal in $\eosos{n}(\instr, \sigma,f_0)$ by a similar proof to \cite[Lemma 4.3]{hoffman2024density}.

\section{Comb percolation and odometers} \label{sec:mulpo}

In this section we describe how comb percolation can be sampled using instructions and the minimal odometer. In addition, we describe a correspondence between infection paths and odometers.

\subsection{Coupling infection sets and odometers}\label{sec:cpcoupling} 
As described in Section \ref{sec:cpdef} comb percolation is characterized by the sequence of shapes which are determined by the values $U_j, R_{j, i}, T_{j, i},$ and $S_{j,i,\ell}$. Take some initial configuration $\sigma$ on $\comb{n}$, instructions $\instr$, and flow $f_0$. Let $\m$ be the minimal odometer of $\eosos{n}(\instr,\sigma,f_0)$. For fixed $v\in \ii{1, n}$ consider the instruction lists 
\begin{align}
    &\instr_v\bigl(\m(v)\bigr),\, \instr_v\bigl(\m(v)+1\bigr),\,\instr_v\bigl(\m(v)+2\bigr),\ldots,\\
    &\instr_{v'}\bigl(\m(v')\bigr),\, \instr_{v'}\bigl(\m(v')+1\bigr),\,\instr_{v'}\bigl(\m(v')+2\bigr),\ldots,
\end{align}
which always start with $\Left$ and $\Down$ respectively. We may partition the spine instruction stack into slots $0,1,2,\ldots$ where each slot begins with a $\Left$ instruction and includes the string of instructions before the next $\Left$ (see the example at the end of the section). We begin the correspondence by setting $U_j$ for $j \ge 0$ to be one plus the number of $\Up$ instructions in slot $j$. This further divides slot $j$ into $U_j$ chunks, each beginning with an $\Up$ or $\Left$ instruction and ending before the next $\Up$ or $\Left$. We define $R_{j,i}$ as one more than the number of $\Right$ instructions within the $i^\text{th}$ chunk. 

We pair the sequence of $\Up$ instructions in the spine stack with the $\Down$s in the tooth stack (excluding the first $\Down$ at the start of the list), and then partition the tooth instruction stack into a corresponding set of slots and chunks, with an $\Up$ instruction going into the chunk corresponding to its pair's. All $\Sleep$ instructions in the tooth stack go into the same chunk as the $\Down$ instruction immediately to the right. Note that the first chunk of a tooth slot is always empty, since the corresponding spine slot has a $\Left$ instruction not a $\Down$. For $2\le i\le U_j$, we define $T_{j,i}$ to be the indicator of the presence of a $\Sleep$ instruction directly before the corresponding $\Down$ instruction in the tooth chunk. For $j \neq 0$ define $T_{j, 1} := T_{j-1, U_{j-1}}$, and define $T_{0, 1}$ to be the indicator of the presence of a $\Sleep$ instruction directly before the $\Down$ instruction $\instr_{v'}(\m(v'))$. 
Within the $i^\text{th}$ spine chunk there are a further $R_{j,i}+1$ regions between any successive jump instructions. For the $\ell^\text{th}$ region we define $S_{j,i,\ell}$ to be the indicator of the presence of a $\Sleep$ instruction within this region.  

Carrying this out for each $v, v'$ gives the \textit{coupled construction} $\mathcal I(\instr, \sigma, f_0)$  of comb percolation that is identical in law to the infection set $\mathcal I$ of $\lawC$-layer percolation for steps $0$ to $n$. The following correspondence $\Phi$ takes $u\in\eosos{n}(\instr,\sigma,f_0)$ and maps it to the sequence of cells
 $\bigl((r_v,s_v)_v,\,0\leq v\leq n\bigr)$ in $\mathcal I (\instr, \sigma, f_0)$ given by
\begin{align}
  r_v &= \rt{u}{v} - \rt{\m}{v},\label{eq:cor}\\
  s_v &= \sum_{i=1}^v\1\{\instr_i(u(i))=\Sleep \}+\1\{\instr_{i'}(u(i'))=\Sleep \}.
\end{align}

\begin{prop}\thlabel{prop:surj}
The map $\Phi$ is surjective.   
\end{prop}
The proof is similar to that of \cite[Proposition 4.6]{hoffman2024density}. We provide a sketch in the appendix.
  
\subsection{Example}
We provide an example demonstrating how to generate shapes in step $k$ of comb percolation. Suppose the instructions at sites $k'$ and $k$, respectively, are:
\begin{align*}
&
\textcolor{black}{
\underbrace{\cdots \mathtt{\mid \: \tikzdashedmid D \tikzdashedmid SD\mid \:\: \mid \: \tikzdashedmid SD \tikzdashedmid SSD \mid \: \tikzdashedmid D \tikzdashedmid SD}}_{\texttt{ignore}}
\mid
\underbrace{\: \tikzdashedmid \texttt{SD}}_{\shape{0}} \mid
\underbrace{\: \tikzdashedmid \texttt{D}}_{\shape{1}} \mid
\underbrace{\: \tikzdashedmid \texttt{D} \tikzdashedmid \texttt{SD}}_{\shape{2}} \mid
\cdots
}
\\
&
\textcolor{black}{
\underbrace{\cdots \mathtt{\mid LR\tikzdashedmid US \tikzdashedmid U\mid LSRSS\mid L \tikzdashedmid U \tikzdashedmid URSR\mid L\tikzdashedmid USRS\tikzdashedmid U}}_{\texttt{ignore}}
\mid
\underbrace{\texttt{LSR}\tikzdashedmid \texttt{URRS}}_{\shape{0}} \mid
\underbrace{\texttt{LSSRRS} \tikzdashedmid \texttt{UR}}_{\shape{1}} \mid
\underbrace{\texttt{LS} \tikzdashedmid \texttt{USSS} \tikzdashedmid \texttt{USR}}_{\shape{2}} \mid
\cdots
}
\end{align*}
where we have ignored instructions as dictated by the minimal odometer. Vertical bars indicate slots associated to $\Left$ instructions in the spine stack, and dashed vertical bars further indicate the chunks determined by $\Up$ instructions between successive $\Left$ instructions.

Given these instructions, $\shape{0}^k$ has $U_0=2$ chunks. The first chunk, shown in red below, has width $R_{0,1} = 2$. Its base is height two since $T_{0,1} = 1$, inherited from the previous $\Down$ instruction and $1$ because it is preceded by a $\Sleep$. The first jump in the spine chunk is followed by a $\Sleep$, while the second is not, so $S_{0,1,1}=1$ and $S_{0,1,2}=0$, determining which columns are one extra in height.

The second chunk, in orange below, has width $R_{0,2}=3$. It also has base height two from the presence of a $\Sleep$ in the tooth chunk (and thus so will the first chunk of the next shape, with the inheritance persisting until there is a second chunk). The extra-high columns are determined by $S_{0,1,1}=0,S_{0,1,2}=0,S_{0,1,3}=1$ based on the presence of $\Sleep$ instructions in the spine chunk.

The chunks are 
\begin{center}
    \begin{tikzpicture}[scale=.3]
    \fillcellwithoffset{red}{(0,0)}  
    \fillcellwithoffset{red}{(0,1)}  
    \fillcellwithoffset{red}{(0,2)}
    \fillcellwithoffset{red}{(1,0)}  
    \fillcellwithoffset{red}{(1,1)}  

    \draw[gray] (0,0) grid[step=1] (4,3);
\end{tikzpicture}
\quad 
\begin{tikzpicture}[scale=.3]

    \fillcellwithoffset{orange}{(0,0)}  
    \fillcellwithoffset{orange}{(0,1)}  
    \fillcellwithoffset{orange}{(1,0)}  
    \fillcellwithoffset{orange}{(1,1)}
    \fillcellwithoffset{orange}{(2,0)}  
    \fillcellwithoffset{orange}{(2,1)}
    \fillcellwithoffset{orange}{(2,2)}
    \draw[gray] (0,0) grid[step=1] (4,3);
\end{tikzpicture}

\end{center}
and so $\shape{0}^k$ is 
\begin{center}
    \begin{tikzpicture}[scale=.3]
    \fillcellwithoffset{red}{(0,0)}  
    \fillcellwithoffset{red}{(0,1)}  
    \fillcellwithoffset{red}{(0,2)}
    \fillcellwithoffset{red}{(1,0)}  
    \fillcellwithoffset{red}{(1,1)}  

    \fillcellwithoffsetdiag{orange}{(1,0)}  
    \fillcellwithoffsetdiag{orange}{(1,1)}  
    \fillcellwithoffset{orange}{(2,0)}  
    \fillcellwithoffset{orange}{(2,1)}
    \fillcellwithoffset{orange}{(3,0)}  
    \fillcellwithoffset{orange}{(3,1)}
    \fillcellwithoffset{orange}{(3,2)}
    \draw[gray] (0,0) grid[step=1] (4,3);
\end{tikzpicture}

\end{center} 
Similarly, $\shape{1}^k$ and $\shape{2}^k$ are
\begin{center}
    \begin{tikzpicture}[scale=.3]
    \fillcellwithoffset{yellow}{(0,0)}  
    \fillcellwithoffset{yellow}{(0,1)}  
    \fillcellwithoffset{yellow}{(0,2)}  
    \fillcellwithoffset{yellow}{(1,0)}  
    \fillcellwithoffset{yellow}{(1,1)}  
    \fillcellwithoffset{yellow}{(2,0)}  
    \fillcellwithoffset{yellow}{(2,1)}  
    \fillcellwithoffset{yellow}{(2,2)}
    \fillcellwithoffsetdiag{green}{(2,0)}  
    \fillcellwithoffset{green}{(3,0)}
    
    \draw[gray] (0,0) grid[step=1] (4,3);
\end{tikzpicture}
\quad 
\begin{tikzpicture}[scale=.3]

    \fillcellwithoffset{blue}{(0,0)}
    \fillcellwithoffset{blue}{(0,1)}
    \fillcellwithoffsetdiag{purple}{(0,0)}
    \fillcellwithoffsetdiag{purple}{(0,1)}
    \fillcellwithoffsetnegdiag{lightgray}{(0,0)}
    \fillcellwithoffsetnegdiag{lightgray}{(0,1)}
    \fillcellwithoffset{lightgray}{(0,2)}
    \fillcellwithoffset{lightgray}{(1,0)}
    \fillcellwithoffset{lightgray}{(1,1)}
    \draw[gray] (0,0) grid[step=1] (4,3);
\end{tikzpicture}
\end{center}

\section{Proof of \texorpdfstring{\thref{thm:cp}}{Theorem~\ref{thm:cp}}} \label{sec:cp} 
 This proof relies on several results from \cite{hoffman2024density} that extend to the comb. Their statements  and explanations for their continued applicability are in the appendix. 
Let $\mathbf 1_{\combinterior{n}}$ be the configuration with one active particle on each non-sink vertex, whose stabilization has the stationary distribution \cite{levine2021exact}. Consider the net flow $f_0= -\lfloor(1-\f \rho 2 ) n\rfloor$ for arbitrary $0<\rho < \critstarc$. Let $\mathcal O_n = \eosos{n}(\instr,\mathbf 1_{\combinterior{n}}, f_0)$ with minimal odometer $\m$. We will say that a sequence of events $A_n$ occurs \textit{with overwhelming probability} (w.o.p.) if there are constants $c,C>0$ depending only on $\rho$ such that $\P(A_n) \geq 1-Ce^{-cn}$ for all large $n$.

Let $k$ be such that $\rho_{\mathtt{COMB}}^{(k)} > \rho$. Let $u \in \mathcal O_n$ be an odometer corresponding to the $k$-greedy infection path $((r_t,s_t)_t)_{0 \leq t \leq n}$ in $\mathcal I(\instr ,\sigma,f_0)$ under the correspondence in \thref{prop:surj}. Proposition \ref{prop:greedy-path-size} gives estimates on the magnitude of $r_n$ and Proposition \ref{prop:min-odometer-concentration} gives estimates on the magnitude of $\rt{\m}{n}$. Together these estimates imply that $r_n \geq \rt{\m}{n}$ w.o.p. Using the correspondence at \eqref{eq:cor}, we deduce that $u(n) \geq 0$ w.o.p. Also, $u(1) \ge 0$ necessarily since $f_0 \le 0$. Combining Lemma \ref{lem:flow} for the special case $\sigma = \mathbf 1_{\combinterior{n}}$ and Lemma \ref{lem:uniflow}, we deduce that as long as $u(1)$ and $u(n)$ are nonnegative, then so are $u(v)$ for all $v$ along the spine. This in turn implies that $u(v')\geq 0$ for all teeth vertices as well.

As $u$ is a nonnegative stable odometer, we may then apply the least-action principle from Lemma \ref{lem:dual.lap} to deduce that the true odometer that stabilizes $\sigma$ uses fewer instructions at site $1$. In particular, it sends no more than $|f_0|$ particles to the sink at $0$ w.o.p. By symmetry, the same argument could be applied to deduce that the true odometer sends no more than $|f_0|$ particles to the sink at $n+1$ w.o.p., and hence, the true odometer leaves behind $S_n \geq 2n - 2|f_0| \geq  \rho n$ particles w.o.p. It follows that $\liminf \E[S_n]/n \geq \rho$, which implies our result as $\rho$ was arbitrary.

\section{Proof of Theorem \ref{thm:bound}} \label{sec:bound}

\subsection{Lower bound}
We describe an iterative scheme to build an infinite infection path $\gamma = \left(\mathbf{c}_k\right)_{k \ge 0}$ starting from $\mathbf{0}$ in comb percolation, with $\mathbf{c}_k = (r_k, s_k)_k$. The path will be the concatenation of paths $\gamma_0, \gamma_1, \ldots$ that form a renewal process in the following sense. Letting $K_0 < K_1 < \ldots$ be such that $\gamma_n$ starts at cell $\mathbf{c}_{K_n}$ and ends at $\mathbf{c}_{K_{n+1}}$, the $(K_{n+1} - K_n, s_{K_{n+1}} - s_{K_n})_{n \ge 0}$ will be i.i.d., thus forming what is termed a renewal process with reward $s_n$. 

Take $n \ge 0$, and suppose that we have constructed $\gamma_0, \ldots, \gamma_{n-1}$. The cell $\mathbf{c}_{K_n}$ is known then, as either the final cell of the previous path, or $\mathbf{0}$ when $n = 0$. To form $\gamma_n$, we mostly restrict to the sub-shapes corresponding to interval percolation with parameter $3\lambda/2$ given by \thref{lemma:IncreasedSleepRateCoupling}, for which any cell infected in this coupled subprocess is also infected in the complete comb percolation. For our first attempt at forming $\gamma_n$, select the $1$-greedy path from those cells infected by $\mathbf{c}_{K_n}$ in the restricted subprocess. The attempt succeeds if the $T$-indicator variable $T_{j,i}$ associated to the final infection of this path (see Section~\ref{sec:cpdef}) is $1$. If not, then we select the $2$-greedy path and check its final infection's $T$-indicator, stopping at the first $K$ where the $T$-indicator is $1$. We set $\gamma_n$ to be this $K$-greedy path with the modification that the terminal cell is one row higher, at the cell corresponding to the positive $T$-indicator. We then have $K_{n+1} - K_n = K$.

Since the $T$-indicators are independent between steps and are independent of the subshapes being used, we have $K_{n+1} - K_n \sim \Geo(\lambda/(1+\lambda))$ and $s_{K_{n+1}} - s_{K_n}$ conditional on $K_{n+1} - K_n=k$ is distributed as $X_k + 1$, where $X_k$ is the height of the $k$-greedy path of interval percolation. Also, the $(K_{n+1} - K_n, s_{K_{n+1}} - s_{K_n})_{n \ge 0}$ are indeed independent, again by the independence of infections across different steps. Therefore, by the theory of renewal reward processes (see for instance \cite[Theorem 3.6.1]{ross1996stochastic} and the ensuing remark),
 $$\lim_{n \to \infty} \E\left[\f{s_n}n\right] = \f{\E[s_{K_1}]}{\E[K_1]} = \f \lambda { 1 + \lambda} \left( 1 + \E\left[ X_{K_1}\right]\right) = \f \lambda { 1 + \lambda} \left( 1 + \E\left[K_1 \rho_{\mathtt{INTERVAL}}^{(K_1)}\left(3\lambda/2\right)\right]\right).$$
Since our construction uses subshapes of comb percolation, $s_n$ is a lower bound for the growth rate of comb percolation, so the above quantity is bounded by $\critstarc(\lambda)$, which by \thref{thm:cp} is bounded by $\critcomb{\lambda}$.

\subsection{Upper bound}
Again, consider the stabilization of the one-per-site configuration. Call the particles that start on the teeth \textit{teeth-particles} and those that start on the spine \textit{spine-particles}. Using the abelian property, in the first step of the stabilization we may leave the teeth-particles in place and only topple spine-particles until every spine-particle is either asleep on the spine or at a sink. When a particle is sent up to a tooth it is immediately sent down, so $\Up$ instructions have no net effect on the configuration. The activity on the spine then in the first step couples with ARW on the interval with sleep rate $\lambda$ and jump rate $2/3$, or scaling up, jump rate $1$ and sleep rate $\lambda'=3\lambda/2$ (this factor of $3/2$ shows up for similar reasons as in \thref{lemma:IncreasedSleepRateCoupling}). After step one then, at each tooth there is an active particle, and the configuration restricted to the spine, $\tau_n$, exactly follows the stationary distribution of driven-dissipative ARW on $\ii{1,n}$ with parameter $\lambda'$. Stabilizing further can only remove more particles, so \begin{align}
    S_n \le n + |\tau_n|,\label{eq:particles-dont-come-back}
\end{align} where absolute value denotes the total number of particles of a configuration. Then since
\begin{align}
    \lim_{n \to \infty}\E\frac{\left|\tau_n\right|}{n} = \critinterval(\lambda'), \label{eq:density-conj}
\end{align} we conclude that  $$\critcomb{\lambda} = \liminf_{n\to\infty} \frac{\mathbf{E}[S_n]}{n} \leq 1 + \critinterval(\lambda').$$

\section*{Appendix}
\label{sec:appendix}
We state results from \cite{hoffman2024density} and then explain why each one or a close analog continues to hold for comb percolation. Note that their leftmost non-sink vertex is $0$ not $1$, with the left sink at $-1$, so the theorem statements below need to be slightly shifted to fit our perspective.

\begin{prop}[Proposition~5.16] \thlabel{prop:greedy-path-size}
    Let $(0,0)_0=(r_0,s_0)_0\to(r_1,s_1)_1\to\cdots$ be the $k$-greedy infection path in interval percolation. There exist constants $C,c$ depending only on $\lambda$ and $k$ such that for all $n$ and all $t\geq 5$, \begin{align} \P\Biggl(\abs[\bigg]{r_n-\frac{\greedyInterval{k} n^2}{2}}\geq tn^{3/2}\Biggr) &\leq C\exp\biggl(-\frac{ct^2}{1+\frac{t}{\sqrt{n}}}\biggr)\label{eq:greedy.path.u},\\\intertext{and} \P\biggl(\abs[\Big]{s_n-\greedyInterval{k} n}\geq t\sqrt{n}\biggr)&\leq 2e^{-ct^2}.\label{eq:greedy.path.s}
   \end{align} 
\end{prop}

This continues to hold with $\greedyInterval{k}$ replaced with $\greedyComb{k}$ because, just like in interval percolation, the shapes used in each step have independent $\Geo(1/2)$ bases. Thus, a given row in comb percolation is still sandwiched between two critical Geometric branching processes with immigration that depends on the height and width of the row as in \cite[Lemma 5.15]{hoffman2024density}.

\begin{prop}[Proposition~5.8]\thlabel{prop:min-odometer-concentration}
 Let $\m'$ be the minimal odometer of $\eosos{n}'(\instr,\sigma,f_0)$, the set of all {odometers} $u$ on $\ii{0,n+1}$ that are stable on $\ii{1,n}$ under $\sigma,\instr$ and whose net flow from site~$0$ to site~$1$ is equal to  $f_0$ (See \cite{hoffman2024density} for more formal definitions).
 
  Let $e_i= -f_0-\sum_{v=1}^i\abs{\sigma(v)}$
  and suppose that $\abs{e_i}\leq\emax$ for some $\emax\geq 1$.
  For some constants $c,C>0$ depending only on $\lambda$, it holds
  for all $t\geq 4\emax$ and $1 \leq j \leq n$ that
\begin{align}\label{eq:min.odometer.bound}
    \P\Biggl(\abs[\bigg]{\rt{\m'}{j}-\sum_{i=1}^j e_i} \geq t\Biggr)
      \leq C\exp\biggl(-\frac{ct^2}{n\bigl(n\emax +t\bigr)}\biggr).
\end{align}
\end{prop}

This continues to hold for $j$ on the spine of $\comb{n}$ and with $e_i$ replaced by $-f_0 - \sum_{v=1}^i \abs{\sigma(v)} + \abs{\sigma(v')}$. This is because the comb percolation minimal odometer $\m \in \eosos{n}(\instr, \sigma, f_0)$, like $\m'$, only depends on the counts of $\Left$ and $\Right$ instructions and the configuration over each $v \in \ii{1,n}$. Since there are still a $\Geo_0(1/2)$-distributed number of $\Right$ instructions between subsequent $\Left$ instructions, the growth of $\m$ behaves like that of $\m'$ (with the augmented configuration counts). 

\begin{lemma}[Lemma~4.1]\thlabel{lem:flow}
  Let $u$ be an odometer on $\ii{0,n}$.
  Let $f_v=\rt{u}{v}-\lt{u}{v+1}$, the net flow from $v$ to $v+1$.
  Let $s_v=\sum_{i=1}^v\1\{\instr_i(u(i))=\Sleep\}$.
  Then $u$ is stable on $\ii{1,n-1}$ for initial configuration $\sigma$ if and only if
  \begin{align}\label{eq:lemflow}
    f_v = f_0 + \sum_{i=1}^v\abs{\sigma(i)} - s_v\quad\text{for all $v\in\ii{0,n-1}$.}
  \end{align}
\end{lemma}

This lemma still holds true on the comb with one minor modification. Namely for $u$ an odometer on $\comb{n}$, $f_v$ still the net flow from $v$ to $v+1$, $d_{v'}$ the net flow from $v'$ to $v$, and $s_v:=\sum_{i=1}^v\1\{\instr_i(u(i))=\Sleep\}+\1\{\instr_{i'}(u(i'))=\Sleep\}$, $u$ is stable on $\comb{1,n}$, the subsection of the spine $\ii{1,n}$ along with the corresponding teeth, if and only if 
\begin{align}
    f_v = f_0 + \sum_{i=1}^v\abs{\sigma(i)}+\abs{\sigma(i')} - s_v
  \end{align}
  and
  \begin{align}
    d_{v'} = \abs{\sigma(v')}-\1\{\instr_{v'}(u(v'))\}
  \end{align}
for all $v\in\ii{1,n}$.

\begin{lemma}[Lemma~8.3]\thlabel{lem:uniflow}
  Let $u$ be an odometer on $\ii{0,n}$ stable on $\ii{1,n-1}$ with $u(0)\geq 0$
  and $u(n)\geq 0$.
  Let $f_v=\rt{u}{v}-\lt{u}{v+1}$.
  Suppose there exists $k$ such that $f_v\leq 0$ for $0\leq v<k$
  and $f_v\geq 1$ for $k\leq v\leq n$. Then $u(v)\geq 0$ for all $v\in\ii{0,n}$.  
\end{lemma}

This lemma still holds on the comb through identification of $\ii{1,n}$ and the spine of $\comb{n}$. Additionally, if $u(v)\ge 0$ then $u(v')\ge 0$ for a stable odometer $u$, and so these conditions ensure that $u$ is nonnegative on all of $\comb{n}$.   

\begin{lemma}[Lemma~2.5]\thlabel{lem:dual.lap}
  Let $u'$ be an odometer on $\ii{a,b}$ that is stable on $\ii{a,b}$, for given initial configuration
  and instruction lists on the interval.
  The true odometer stabilizing $\ii{a,b}$ leaves at least as many
  particles sleeping on $\ii{a,b}$ as does $u'$.
\end{lemma}

The analog of this lemma still holds on the comb. The true odometer stabilizing $\comb{a,b}$ leaves at least as many particles on $\comb{a,b}$ as any stable odometer $u'$. This holds on all finite connected graphs with sinks by summing $h(v)$ (defined in Section \ref{sec:generalstableodometers}) over all sites, which telescopes to leave only the flow into sinks, which must be greater than the true stabilizing odometer by the least-action principle. 

\begin{proof}[Proof sketch of \thref{prop:surj}]
  The proof is similar to that of \cite[Proposition 4.6]{hoffman2024density}. Let $((r_v, s_v)_v, 0\le v\le n)\in \mathcal I (\instr, \sigma, f_0)$ be an infection path. We will construct $u$ which is the preimage of this infection path inductively. First, define $u(0)=0$, and assume that $u(v-1)$ has been defined. We let $i_v$ be the minimal integer such that, beginning after instruction $\m(v)$, $r_{v-1}+s_{v-1}$ $\Left$ instructions and $r_v$ $\Right$ instructions have occurred in the spine instruction stack at $v$. Such an $i_v$ exists because of our coupling of instruction stacks with comb percolation, as we know $(r_{v-1}, s_{v-1})_{v-1}$ infects $(r_v, s_v)_v$. Similarly let $i_{v'}$ be the minimal integer such that, beginning after instruction $\m(v')$, $\up{v}{i_v}-\up{v}{\m}$ $\Down$ instructions have occurred in the tooth instruction stack at $v'$. 
  
  We now break into cases: $s_{v}-s_{v-1}=0$, $s_v-s_{v-1}=1$, and $s_v-s_{v-2}=2$. If $s_{v}-s_{v-1}=0$, we define $u(v)=i_v$ and $u(v')=i_{v'}$. If $s_{v}-s_{v-1}=1$ and there exists a $\Sleep$ instruction in the spine stack after $i_v$ but before the next $\Right$ or $\Left$ instruction, with some minimal index $j_v$, we let $u(v)=j_v$, $u(v')=j_{v'}$, where $j_{v'}$ is the minimal index after $i_{v'}$ for which the number of $\Down$ instructions between $i_{v'}$ and $j_{v'}$ matches the number of $\Up$ instructions executed between $i_v$ and $j_v$ in the spine stack. If no such index exists, by our coupling we know that by pairing $\Up$ instructions (including the $0^\text{th}$ instruction with $i_{v'}$) within the spine stack between $i_v$ and the next $\Right$ or $\Left$ instruction with $\Down$ instructions within the tooth stack after $i_{v'}$, there exists some minimal index $\ell_{v'}$ corresponding to a paired $\Down$ instruction such that directly before $\ell_{v'}$ there is a $\Sleep$ instruction. Letting $k_v$ be the index of the paired $\Up$ instruction in the spine stack, we define $u(v)=k_v$ and $u(v')=\ell_{v'}-1$.

  Finally, if $s_{v}-s_{v-1}=2$, again by the coupling of infection paths and instruction stacks, we know that there exists some minimal index $k_v$ (potentially $i_v$) between the indices $i_v$ and the index of the next $\Right$ or $\Left$ instruction such that: $k_v$ indexes an $\Up$ instruction (if $k_v\neq i_v$), the paired $\Down$ instruction in the tooth stack with index $\ell_{v'}$ has a $\Sleep$ instruction directly proceeding it, and $k_v+1$ indexes a $\Sleep$ instruction. We then define $u(v)=k_v+1$, $u(v')=\ell_{v'}-1$. 
  It can be checked that such a $u$ is stable and $\Phi(u)$ is exactly the infection path $((r_v, s_v)_v, 0\le v\le n)$.   
\end{proof}

\bibliographystyle{alpha}
\bibliography{references}
\end{document}